\documentclass[12pt,reqno]{article}

\usepackage[usenames]{color}
\usepackage{amssymb}
\usepackage{graphicx}
\usepackage{amscd}

\usepackage{amsthm}
\newtheorem{theorem}{Theorem}

\newtheorem{proposition}[theorem]{Proposition}
\newtheorem{corollary}[theorem]{Corollary}

\theoremstyle{definition}

\newtheorem{example}[theorem]{Example}

\usepackage[colorlinks=true,
linkcolor=webgreen, filecolor=webbrown,
citecolor=webgreen]{hyperref}

\definecolor{webgreen}{rgb}{0,.5,0}
\definecolor{webbrown}{rgb}{.6,0,0}

\usepackage{color}

\usepackage{float}

\usepackage{graphics,amsmath,amssymb}
\usepackage{amsfonts}
\usepackage{latexsym}
\usepackage{epsf}

\newcommand{\seqnum}[1]{\href{http://www.research.att.com/cgi-bin/access.cgi/as/~njas/sequences/eisA.cgi?Anum=#1}{\underline{#1}}}

\begin{document}

\begin{center}
\vskip 1cm{\LARGE\bf On the Group of Almost-Riordan Arrays} \vskip 1cm \large
Paul Barry\\
School of Science\\
Waterford Institute of Technology\\
Ireland\\
\href{mailto:pbarry@wit.ie}{\tt pbarry@wit.ie}
\end{center}
\vskip .2 in

\begin{abstract} We study a super group of the group of Riordan arrays, where the elements of the group are given by a triple of power series. We show that certain subsets are subgroups, and we identify a normal subgroup whose cosets correspond to Riordan arrays. We give an example of an almost-Riordan array that has been studied in the context of Hankel and Hankel plus Toepliz matrices, and we show that suitably chosen almost-Riordan arrays can lead to transformations that have interesting Hankel transform properties. \end{abstract}

\section{Introduction}
The group of Riordan arrays $\mathcal{R}$ \cite{SGWW}  was first introduced by Shapiro, Getu, Woan, and Woodson in the early $1990$'s. Since then, they have been extensively studied and applied in a number of different fields.
At its simplest, a Riordan array is formally defined by a pair of power series, say $g(x)$ and $f(x)$, where $g(0)=1$ and $f(x)=x+a_2 x^2+a_3x^2+\ldots$, with integer coefficients (such Riordan arrays are called ``proper'' Riordan arrays). The pair $(g, f)$ is then associated to the lower-triangular invertible matrix whose $(n,k)$-th element $T_{n,k}$ is given by
$$T_{n,k}=[x^n] g(x)f(x)^k.$$
We sometimes write $(g(x), f(x))$ although the variable ``$x$'' here is a dummy variable, in that
$$T_{n,k}=[x^n] g(x)f(x)^k = [t^n] g(t)f(t)^k.$$
In this paper, we shall define a group of matrices defined by a triple of power series, and we shall demonstrate that the Riordan group is a factor group of this new group.

Having defined products and inverses in this new group, we look at some additional properties, and give some examples. Our first example in this section is based on a special transformation that has been studied in the context of Hankel plus Toeplitz matrices \cite{Toeplitz, Basor}

By looking at examples closely related to the Catalan numbers, we arrive at transformations of sequences that have interesting Hankel transform properties. Here we recall that for a sequence $a_n$ we define its Hankel transform to be the sequence of determinants $h_n=|a_{i+j}|_{0 \le i,j \le n}$.

All the power series and matrices that we shall look at are assumed to have integer coefficients. Thus power series are elements of $\mathbb{Z}[[x]]$. The generating function $1$ generates the sequence that we denote by $0^n$, which begins $1,0,0,0,\ldots$. All matrices are assumed to begin at the $(0,0)$ position, and to extend infinitely to the right and downwards. Thus matrices in this article are elements of $\mathbb{Z}^{\mathbb{N}_0 \times \mathbb{N}_0}$. When examples are given, an obvious truncation is applied.

The Fundamental Theorem of Riordan arrays \cite{Survey} says that the action of a Riordan array on a power series, namely
$$ (g(x), f(x))\cdot a(x)= g(x)a(f(x)),$$ is realised in matrix form by
$$\left(T_{n,k}\right)\left( \begin{array}{c}a_0\\a_1\\a_2\\a_3\\ \vdots\end{array}\right)
=\left( \begin{array}{c}b_0\\b_1\\b_2\\b_3\\ \vdots\end{array}\right),$$ where
the power series $a(x)$ expands to give the sequence $a_0, a_1, a_2, \ldots$, and the image sequence $b_0, b_1, b_2, \ldots$ has generating function $g(x)a(f(x))$.

An important feature of Riordan arrays is that they have a number of sequence characterizations \cite{Cheon, He}. The simplest of
these
is as follows.
\begin{proposition} \label{Char} \cite[Theorem 2.1, Theorem 2.2]{He} Let $D=[d_{n,k}]$ be an infinite triangular matrix. Then $D$ is a Riordan array if and only if there
exist two sequences $A=[a_0,a_1,a_2,\ldots]$ and $Z=[z_0,z_1,z_2,\ldots]$ with $a_0 \neq 0$, $z_0 \neq 0$ such that
\begin{itemize}
\item $d_{n+1,k+1}=\sum_{j=0}^{\infty} a_j d_{n,k+j}, \quad (k,n=0,1,\ldots)$
\item $d_{n+1,0}=\sum_{j=0}^{\infty} z_j d_{n,j}, \quad (n=0,1,\ldots)$.
\end{itemize}
\end{proposition}
The coefficients $a_0,a_1,a_2,\ldots$ and $z_0,z_1,z_2,\ldots$ are called the $A$-sequence and the $Z$-sequence of the Riordan array
$D=(g(x),f(x))$, respectively.
Letting $A(x)$ be the generating function of the $A$-sequence and $Z(x)$ be the generating function of the $Z$-sequence, we have
\begin{equation}\label{AZ_eq} A(x)=\frac{x}{\bar{f}(x)}, \quad Z(x)=\frac{1}{\bar{f}(x)}\left(1-\frac{1}{g(\bar{f}(x))}\right).\end{equation}
Here, $\bar{f}(x)$ is the series reversion of $f(x)$, defined as the solution $u(x)$ of the equation
$$f(u)=x$$ that satisfies $u(0)=0$.

The inverse of the Riordan array $(g, f)$ is given by
$$(g(x), f(x))^{-1}=\left(\frac{1}{g(\bar{f}(x))}, \bar{f}(x)\right).$$

For a Riordan array $D$, the matrix $P=D^{-1}\cdot \overline{D}$ is called its \emph{production matrix}, where $\overline{D}$ is the matrix $D$ with its top row removed.

The
concept of a \emph{production matrix} \cite{ProdMat_0,
ProdMat}
is a general one, but for this work we find it convenient to
review it in
the context of Riordan arrays. Thus let $P$ be an infinite
matrix (most often it will have integer entries). Letting
$\mathbf{r}_0$
be the row vector
$$\mathbf{r}_0=(1,0,0,0,\ldots),$$ we define $\mathbf{r}_i=\mathbf{r}_{i-1}P$, $i \ge 1$.
Stacking these rows leads to another infinite matrix which we
denote by
$A_P$. Then $P$ is said to be the \emph{production matrix} for
$A_P$.

\noindent If we let $$u^T=(1,0,0,0,\ldots,0,\ldots)$$ then we
have $$A_P=\left(\begin{array}{c}
u^T\\u^TP\\u^TP^2\\\vdots\end{array}\right)$$ and
$$\bar{I}A_P=A_PP$$ where $\bar{I}=(\delta_{i+1,j})_{i,j \ge 0}$ (where
$\delta$ is the usual Kronecker symbol):
\begin{displaymath} \bar{I}=\left(\begin{array}{ccccccc} 0 & 1
& 0 & 0 & 0 & 0 & \ldots \\0 & 0 & 1 & 0 & 0 & 0 & \ldots \\
0 & 0 & 0 & 1
& 0 & 0 & \ldots \\ 0 & 0 & 0 & 0 & 1 & 0 & \ldots \\ 0 & 0 &
0
& 0 & 0 & 1 & \ldots \\0 & 0  & 0 & 0 & 0 & 0 &\ldots\\ \vdots
& \vdots &
\vdots & \vdots & \vdots & \vdots &
\ddots\end{array}\right).\end{displaymath}
We have
\begin{equation}P=A_P^{-1}\bar{I}A_P.\end{equation} Writing
$\overline{A_P}=\bar{I}A_P$, we can write this equation as \begin{equation}P=A_P^{-1}\overline{A_P}.\end{equation} Note that $\overline{A_P}$ is $A_P$ with the first row removed.

The production matrix $P$ is sometimes \cite{P_W, Shapiro_bij} called the Stieltjes matrix $S_{A_P}$
associated to $A_P$. Other examples of the use of production matrices can be found in \cite{Arregui}, for instance.

\noindent The sequence formed by the row sums of $A_P$ often
has combinatorial significance and is called the sequence
associated to $P$. Its general
term $a_n$ is given by $a_n = u^T P^n e$ where
$$e=\left(\begin{array}{c}1\\1\\1\\\vdots\end{array}\right).$$
In the context of Riordan
arrays, the production matrix associated to a proper Riordan
array takes on a special form\,:
 \begin{proposition} \label{RProdMat}
\cite[Proposition 3.1]{ProdMat}\label{AZ} Let $P$ be
an infinite production matrix and let $A_P$ be the matrix
induced by $P$. Then $A_P$ is an (ordinary) Riordan matrix if
and only if $P$ is
of the form \begin{displaymath} P=\left(\begin{array}{ccccccc}
\xi_0 & \alpha_0 & 0 & 0 & 0 & 0 & \ldots \\\xi_1 & \alpha_1 &
\alpha_0 & 0 &
0 & 0 & \ldots \\ \xi_2 & \alpha_2 & \alpha_1 & \alpha_0 & 0 &
0 & \ldots \\ \xi_3 & \alpha_3 & \alpha_2 & \alpha_1 &
\alpha_0
& 0 & \ldots
\\ \xi_4 & \alpha_4 & \alpha_3 & \alpha_2 & \alpha_1 &
\alpha_0
& \ldots \\\xi_5 & \alpha_5  & \alpha_4 & \alpha_3 & \alpha_2
&
\alpha_1
&\ldots\\ \vdots & \vdots & \vdots & \vdots & \vdots & \vdots
&
\ddots\end{array}\right),\end{displaymath} where $\xi_0 \neq 0$, $\alpha_0 \neq 0$. Moreover, columns $0$
and $1$ of
the matrix $P$ are the $Z$- and $A$-sequences,
respectively, of the Riordan array $A_P$. \end{proposition}

We shall use the notation $\tilde{a}(x)=\sum_{n=1}a_n x^n=\frac{a(x)-a_0}{x}$ in the sequel, where $a(x)=\sum_{n=0}^n a_n x^n$.

Where possible, we shall refer to known sequences and triangles by their OEIS numbers \cite{SL1, SL2}. For instance, the Catalan numbers $C_n=\frac{1}{n+1}\binom{2n}{n}$ with g.f. $c(x)=\frac{1-\sqrt{1-4x}}{2x}$ is the sequence \seqnum{A000108}, the Fibonacci numbers are \seqnum{A000045}, and the Motzkin numbers $M_n=\sum_{k=0}^{\lfloor \frac{n}{2} \rfloor}\binom{n}{2k}C_k$  are \seqnum{A001006}.

The binomial matrix $B=\left(\binom{n}{k}\right)$ is \seqnum{A007318}. As a Riordan array, this is given by
$$B= \left(\frac{1}{1-x}, \frac{x}{1-x}\right).$$

Note that in this article all sequences $a_n$ that have $a_0 \ne 0$ are assumed to have $a_0=1$. Likewise for sequences $b_n$ with $b_0=0$ and $b_1 \ne 0$, we assume that $b_1=1$.

\section{Definitions and Properties}

An \emph{almost-Riordan array} is defined by an ordered triple $(a,g,f)$ of power series where
$a(x)=\sum_{n=0}^{\infty}a_nx^n ,$ with $a_0=1$, $g(x)=\sum_{n=0}^{\infty} g_n x^n$, with $g_0=1$, and
$f(x)=\sum_{n=0}^{\infty} f_n x^n$, with $f_0=0,\,f_1=1$. The array is identified with the lower-triangular matrix defined as follows: its first column is given by the expansion of $a(x)$, while its first row is the expansion of $1$. The remaining elements of the infinite tri-diagonal matrix (starting at the $(1,1)$ position) coincide with the Riordan array $(g,f)$. Here, we address the first element of the matrix as the $(0,0)$-th element.

We shall denote by $a\mathcal{R}$ the set of almost-Riordan arrays. Formally this is the set of ordered triples $(a,g,f)$ as described above. We identify these triples with lower-triangular matrices as in the example that follows. We define an action of the element $(a, g, f)$ on the power series $b(x)$ by looking at the action of the corresponding matrix on the column vector given by the expansion of $b(x)$. The result $(a, g, f)\cdot b$ is then the generating function of the sequence encapsulated in the column vector that arises by applying the matrix to the column vector whose elements are given by the expansion of $b(x)$.
\begin{example}
We consider the almost-Riordan array defined by $\left(\frac{1}{1-2x}, \frac{1}{1-x}, \frac{x}{1-x}\right)$. This matrix begins
$$\left(
\begin{array}{ccccccc}
 1 & 0 & 0 & 0 & 0 & 0 & 0 \\
 2 & 1 & 0 & 0 & 0 & 0 & 0 \\
 4 & 1 & 1 & 0 & 0 & 0 & 0 \\
 8 & 1 & 2 & 1 & 0 & 0 & 0 \\
 16 & 1 & 3 & 3 & 1 & 0 & 0 \\
 32 & 1 & 4 & 6 & 4 & 1 & 0 \\
 64 & 1 & 5 & 10 & 10 & 5 & 1 \\
\end{array}
\right).$$ Then
$$\left(\frac{1}{1-2x}, \frac{1}{1-x}, \frac{x}{1-x}\right)\cdot \frac{1}{1-x-x^2}=\frac{1-2x-x^2}{1-5x+7x^2-2x^3}$$ is realised in matrix form by
$$\left(
\begin{array}{ccccccc}
 1 & 0 & 0 & 0 & 0 & 0 & 0 \\
 2 & 1 & 0 & 0 & 0 & 0 & 0 \\
 4 & 1 & 1 & 0 & 0 & 0 & 0 \\
 8 & 1 & 2 & 1 & 0 & 0 & 0 \\
 16 & 1 & 3 & 3 & 1 & 0 & 0 \\
 32 & 1 & 4 & 6 & 4 & 1 & 0 \\
 64 & 1 & 5 & 10 & 10 & 5 & 1 \\
\end{array}\right)\cdot \left(\begin{array}{c}1\\1\\2\\3\\5\\8\\13\\ \end{array}\right)=
\left(\begin{array}{c}1\\3\\7\\16\\37\\87\\208\\ \end{array}\right), $$ where the expansion of the image
$\frac{1-2x-x^2}{1-5x+7x^2-2x^3}$ begins $1,3,7,16,37,\ldots$.

\end{example}

We write $(a,0,0)$ for the matrix whose first column is generated by $a(x)$, with zeros elsewhere.
We then have

$$(a,g,f)=(a,0,0)+(xg, f)$$ as a matrix equality.

The elements of the matrix $M=(a, g,f)$ are easily described. Letting the $(n,k)$-th element of $M$ be denoted by $M_{n,k}$, we have
$$M_{n,k}=[x^{n-1}] g f^{k-1}, \quad \textrm{for\,} n,k \ge 1,\quad M_{n,0}=a_n,\quad M_{0,k}=0^k.$$
\begin{example} The almost-Riordan array $\left(1, \frac{1}{1-x}, \frac{x}{1-x}\right)$ has general term
$\binom{n-1}{n-k}$. In this case, this matrix coincides with the Riordan array $\left(1, \frac{x}{1-x}\right)$.
\end{example}
Our first result is the Fundamental Theorem of almost-Riordan arrays.

\begin{proposition} Let $(a, g, f)$ define an almost-Riordan array, and consider a power series $h(x)=\sum_{n=0}^{\infty}h_n x^n$. We have
\begin{equation} (a, g, f) \cdot h(x)=h_0 a(x)+x g(x) \tilde{h}(f(x)),\end{equation} where
$$\tilde{h}(x)=\frac{h(x)-h_0}{x}.$$
\end{proposition}
\begin{proof}
We have
\begin{eqnarray*}
(a, g, f) \cdot h(x)&=&(a, xg, xgf, xgf^2,\cdots)\cdot \left(\begin{array}{c} h_0\\h_1\\h_2\\h_3\\ \vdots  \end{array}\right)\\
&=& h_0 a + h_1 x g+ h_2 xgf + h_3 x g f^2+\cdots \\
&=& h_0 a + x g(h_1+h_2 f + h_3 f^2+ \cdots) \\
&=& h_o a + xg \tilde{h}(f).\end{eqnarray*}
\end{proof}
\begin{corollary} We have
\begin{equation}(a, g, f)\cdot 1 = a.\end{equation}
\end{corollary}
\begin{proof} This follows since we have
$$(a, g, f)\cdot 1=1.a+xg \tilde{1}(f)=a,$$ since
$\tilde{1}=0$.
\end{proof}

We next define the product of two almost-Riordan arrays. Thus consider the almost-Riordan arrays
$(a, g, f)$ and $(b, u, v)$. We define their product by
\begin{equation} (a, g, f) \cdot (b,u, v)= ( (a, g, f) b, g u(f), v(f)).\end{equation}
By construction, the product of two almost-Riordan arrays is again an almost-Riordan array.

We define $I=(1,1,x)$. We have
\begin{proposition}
$$ I \cdot (b, u, v)=(b,u,v), \quad \quad (a,f,g) \cdot I = (a, f, g).$$
\end{proposition}
\begin{proof}
$$I \cdot (b, u, v)=(1,1,x)\cdot (b,u,v)=((1,1,x) b, 1\, u(x), v(x))=((1,1,x) b, u, v).$$
Now $$(1,1,x) b= b_0 \,1 + x\,1\,\tilde{b}(x)=1 + x \frac{b-1}{x}=b.$$
Thus $$ I \cdot (b, u ,v) = (b, u, v).$$

Now
$$(a, g, f) \cdot I = (a,g,f) \cdot (1,1,x)=((a,g,f)\cdot 1, g\,1(f), x(f))=((a,g,f)\cdot 1, g, f).$$
We have
$$(a, g,f)\cdot 1=1. a+ x g \tilde{1}(f)=a \quad \textrm{since\,} \tilde{1}=0.$$
Thus
$$(a, g, f) \cdot I = (a, g, f).$$

\end{proof}

Thus $I=(1,1,x)$ is an identity element for the set of almost-Riordan arrays. We have elaborated the above proposition to show that the formalism works. A more direct proof is to notice that $I=(1,1,x)$ is of course the normal (infinite) identity matrix with $1$'s on the diagonal and zeros elsewhere.

We now turn to look at the inverse of an almost-Riordan array.
\begin{example} We consider the almost-Riordan array given by $(1, g, f)$. Its inverse is given by
\begin{equation}(1, g, f)^{-1}=\left(1, \frac{1}{g(\bar{f})}, \bar{f}\right).\end{equation} In other words, it is the matrix with first row and columns generated by $1$, and starting at the $(1,1)$-position, it coincides with the inverse Riordan array $(g, f)^{-1}$.
This result is an immediate consequence of standard matrix partitioning.
\end{example}

\begin{proposition} The inverse of the almost-Riordan array $(a, g, f)$ is the almost-Riordan array
\begin{equation} (a, g, f)^{-1}=\left(a^*, \frac{1}{g(\bar{f})}, \bar{f}\right),\end{equation} where
\begin{equation} a^*(x)=(1, -g, f)^{-1} \cdot a(x).\end{equation}
\end{proposition}
\begin{proof}
We need to show that

$$ (a, g, f)\cdot \left(a^*, \frac{1}{g(\bar{f})}, \bar{f}\right)=I,$$ and
$$ \left(a^*, \frac{1}{g(\bar{f})}, \bar{f}\right) \cdot (a, g, f)=I.$$

Clearly, for the first expression, we only need to show that
$$(a, g, f)\cdot a^*=1,$$ since the result then follows by matrix partitioning.
We have
\begin{eqnarray*}
(a, f, g)\cdot a^*&=& (a,g,f)\cdot \left(1, -\frac{1}{g(\bar{f})}, \bar{f}\right)\cdot a(x)\\
&=& \left((a,g,f)\cdot 1, g.-\frac{1}{g(\bar{f}(f))}, \bar{f}(f)\right)\cdot a(x)\\
&=& (a, -1, x) \cdot a\\
&=& a_0 a - x\tilde{a}(x)\\
&=& a-x \frac{a(x)-1}{x}\\
&=& a- a+1\\
&=& 1 \end{eqnarray*}

We must now show that

$$ \left(a^*, \frac{1}{g(\bar{f})}, \bar{f}\right) \cdot (a, g, f)=I.$$
Again, by matrix partitioning, we need to show that
$$ \left(a^*, \frac{1}{g(\bar{f})}, \bar{f}\right)\cdot a =1.$$
We have
\begin{eqnarray*}
\left(a^*, \frac{1}{g(\bar{f})}, \bar{f}\right)\cdot a&=&a_0 a^*+\frac{x}{g(\bar{f})}\tilde{a}(\bar{f})\\
&=& (1, -g, f)^{-1}\cdot a(x)+\frac{x}{g(\bar{f})}\tilde{a}(\bar{f})\\
&=& \left(1, -\frac{1}{g(\bar{f})}, \bar{f}\right)\cdot a(x)+\frac{x}{g(\bar{f})}\tilde{a}(\bar{f})\\
&=& a_0\,1-\frac{x}{g(\bar{f})}\tilde{a}(\bar{f})+\frac{x}{g(\bar{f})}\tilde{a}(\bar{f})\\
&=& 1.\end{eqnarray*}
\end{proof}

Thus the set of almost-Riordan arrays is in fact a group. We denote this group by $a\mathcal{R}$.

The group of Riordan arrays $\mathcal{R}$ is a subgroup of this group, identified as the subgroup of almost-Riordan arrays of the form
$$\left(g, g\frac{f}{x}, f\right).$$
Let us verify that the subset of $a\mathcal{R}$ consisting of arrays of the form $\left(g, g\frac{f}{x}, f\right)$ is closed under the product of $a\mathcal{R}$. Thus let
$\left(g, g\frac{f}{x}, f\right)$ and $\left(u, u \frac{v}{x}, v\right)$ be two elements of this subset.
We have
\begin{eqnarray*}
\left(g, g\frac{f}{x}, f\right) \cdot \left(u, u \frac{v}{x}, v\right)&=&
\left(\left(g, g\frac{f}{x}, f\right) \cdot u, g\frac{f}{x} \frac{u(f) v(f)}{f}, v(f)\right)\\
&=& \left(u_0g+xg \frac{f}{x} \tilde{u}(f), gu(f) \frac{v(f)}{x}, v(f)\right)\\
&=& \left(u_0 g+gf\left(\frac{u(f)-u_0}{f}\right), gu(f) \frac{v(f)}{x}, v(f)\right)\\
&=& \left(u_0 g + g(u(f)-u_0), gu(f) \frac{v(f)}{x}, v(f)\right)\\
&=& \left(gu(f), gu(f) \frac{v(f)}{x}, v(f)\right).\end{eqnarray*}
Thus the subset is closed under products. We next show that this subset is closed under inverses.

We have
\begin{eqnarray*}
\left(g, g\frac{f}{x}, f\right)^{-1}&=& \left(g^*(x), \frac{1}{\frac{gf}{x} \circ \bar{f}}(x),\bar{f}(x)\right)\\
&=&\left(\left(1,-g \frac{f}{x}, f\right)^{-1}\cdot g, \frac{1}{\frac{g(\bar{f}(x))f(\bar{f}(x))}{\bar{f}(x)}}, \bar{f}\right)\\
&=& \left(\left(1, -\frac{1}{\frac{gf}{x}\circ \bar{f}}, \bar{f}\right)\cdot g, \frac{1}{g(\bar{f}(x))}\frac{\bar{f}(x)}{x}, \bar{f}\right)\\
&=& \left(g_0.1-x \frac{1}{\frac{gf}{x} \circ \bar{f}} \tilde{g}(\bar{f}), \frac{1}{g(\bar{f}(x))}\frac{\bar{f}(x)}{x}, \bar{f}\right)\\
&=& \left(g_0-x \frac{1}{g(\bar{f}(x))}\frac{\bar{f}(x)}{x} \left(\frac{g(\bar{f})-g_0}{\bar{f}}\right), \frac{1}{g(\bar{f}(x))}\frac{\bar{f}(x)}{x}, \bar{f}\right)\\
&=& \left(g_0-\frac{1}{g(\bar{f}(x))}(g(\bar{f}(x))-g_0), \frac{1}{g(\bar{f}(x))}\frac{\bar{f}(x)}{x}, \bar{f}\right)\\
&=& \left(\frac{1}{g(\bar{f}(x))}, \frac{1}{g(\bar{f}(x))}\frac{\bar{f}(x)}{x}, \bar{f}\right).\end{eqnarray*}
We thus have
\begin{proposition} The subset of $a\mathcal{R}$ of almost-Riordan arrays of the form
$$\left(g, g\frac{f}{x}, f\right)$$ is a subgroup of $a\mathcal{R}$, isomorphic to the group $\mathcal{R}$ of Riordan arrys. \end{proposition}
\begin{proof} We have just shown that this subset is a subgroup (with identity $(1, 1, x)=\left(1, 1 \frac{x}{x}, x\right)$). There is an obvious $1-1$ correspondence between elements $\left(g, g\frac{f}{x}, f\right)$ of this subgroup and the corresponding element $(g, f) \in \mathcal{R}$. We write $\phi(.)$ for this correspondence so that $$ \phi\left(\left(g, g \frac{f}{x}, f\right)\right)=(g, f).$$ It remains to show that this is a homomorphism.
We have
\begin{eqnarray*}\phi\left( \left(g, g\frac{f}{x}, f\right) \cdot \left(u, u \frac{v}{x}, v\right)\right)&=&
\phi\left( \left(gu(f), gu(f)\frac{v(f)}{x}, v(f)\right)\right)\\
&=& (gu(f), v(f)).\end{eqnarray*}
On the other hand, we have
\begin{eqnarray*}
\phi\left(\left(g, g \frac{f}{x}, f\right)\right)\cdot \phi\left(\left(u, u \frac{v}{x}, v\right)\right)&=&
(g, f) \cdot (u, v) \\
&=& (g u(f), v(f)).\end{eqnarray*}

Similarly, we have
$$\phi\left(\left(g, g\frac{f}{x}, f\right)^{-1}\right)=\phi\left(\left(\frac{1}{g(\bar{f}(x))}, \frac{1}{g(\bar{f}(x))} \frac{\bar{f}}{x}, \bar{f}\right)\right)=\left(\frac{1}{g(\bar{f}(x))}, \bar{f}(x)\right)=(g, f)^{-1}.$$
\end{proof}

This is not the only subgroup of $a\mathcal{R}$ that is isomorphic to $\mathcal{R}$. We have
\begin{proposition} The map
$$\psi: (1, g, f) \mapsto (g, f)$$ is an isomorphism from the subset of $a\mathcal{R}$ comprised of elements of the form $(1, g, f)$ to $\mathcal{R}$.
\end{proposition}
Clearly, we have
$$\psi\left( (1, g, f)\cdot (1, u, v)\right)=\psi\left((1, g u(f), v(f))\right)=(gu(f), v(f))= (g, f)\cdot (u, v),$$ and
$$ \psi\left((1, g, f)^{-1}\right)=\psi\left(\left(1, \frac{1}{g\circ \bar{f}}, \bar{f}\right)\right)=\left(\frac{1}{g\circ \bar{f}}, \bar{f}\right)=(g, f)^{-1}=(\psi(1, g, f))^{-1}.$$

Another subgroup of $a\mathcal{R}$ is the group of almost-Riordan arrays of the form
$$(a, 1, x).$$ We have

$$(a, 1, x)\cdot (b, 1,x)=((a,1,x) b, 1, x),$$ showing that these elements are closed under multiplication.
In fact, we have
\begin{eqnarray*} (a,1,x) \cdot (b,1,x)&=& ((a,1,x) b, 1, x)\\
&=& (b_0 a+x \tilde{b}(x),1,x) \\
&=& \left(b_0 a+x \frac{b(x)-1}{x},1,x\right)\\
&=& (a+b-1, 1,x).\end{eqnarray*}
Turning to inverses, we have
$$(a,1,x)^{-1}=(a^*, 1,x),$$ where
$$a^*=(1,-1,x)\cdot a=a_0.1 -x\tilde{a}(x).$$ Thus the first column of
$(a,1,x)^{-1}$ is given by the sequence
$$a_0, -a_1, -a_2, -a_3, \ldots.$$
\begin{proposition} The subset $\mathcal{N}$ of the group of almost-Riordan arrays $a\mathcal{R}$ of matrices of the form
$(a, 1, x)$ is a subgroup, where products are defined by
$$(a, 1, x) \cdot (b, 1, x)= (a+b-1, 1,x), $$ and inverses are defined by
$$(a, 1, x)^{-1} =(a^*, 1,x)=(a_0.1 -x\tilde{a}(x),1,x).$$
\end{proposition}
\begin{example} The almost-Riordan array $\left(\frac{1}{1-3x}, 1,x\right)$ begins
$$\left(
\begin{array}{ccccccc}
 1 & 0 & 0 & 0 & 0 & 0 & 0 \\
 3 & 1 & 0 & 0 & 0 & 0 & 0 \\
 9 & 0 & 1 & 0 & 0 & 0 & 0 \\
 27 & 0 & 0 & 1 & 0 & 0 & 0 \\
 81 & 0 & 0 & 0 & 1 & 0 & 0 \\
 243 & 0 & 0 & 0 & 0 & 1 & 0 \\
 729 & 0 & 0 & 0 & 0 & 0 & 1 \\
\end{array}
\right).$$ Its inverse begins
$$\left(
\begin{array}{ccccccc}
 1 & 0 & 0 & 0 & 0 & 0 & 0 \\
 -3 & 1 & 0 & 0 & 0 & 0 & 0 \\
 -9 & 0 & 1 & 0 & 0 & 0 & 0 \\
 -27 & 0 & 0 & 1 & 0 & 0 & 0 \\
 -81 & 0 & 0 & 0 & 1 & 0 & 0 \\
 -243 & 0 & 0 & 0 & 0 & 1 & 0 \\
 -729 & 0 & 0 & 0 & 0 & 0 & 1 \\
\end{array}
\right).$$

\end{example}
\begin{proposition} The subgroup $\mathcal{N}$ of $a\mathcal{R}$ of almost-Riordan arrays of the form $(a, 1,x)$ is a normal subgroup of $a\mathcal{R}$.
\end{proposition}
\begin{proof}
We must show that for an arbitrary element $(a, g, f)$ of $a\mathcal{R}$, the element
$$ (a, g, f)\cdot (b, 1,x)\cdot (a, g, f)^{-1}$$ is of the form $(b', 1,x)$ for an appropriate power series $b'$.
We have
\begin{eqnarray*}
(a, g, f)\cdot (b, 1,x)\cdot (a, g, f)^{-1}&=&
((a,g,f)\cdot b, g, f)\cdot \left(a^*, \frac{1}{g(\bar{f})}, \bar{f}\right)\\
&=&\left(((a,g,f)\cdot b, g, f)\cdot a^*, g. \frac{1}{g(\bar{f}(f))}, \bar{f}(f)\right)\\
&=& (((a,g,f) \cdot b, g, f)\cdot a^*, 1, x)
\end{eqnarray*}
as required.
\end{proof}
It is instructive to continue the above calculation. Thus we have

$$(a,g,f) \cdot (b,1,x) \cdot (a, g, f)^{-1}=(((a,g,f) \cdot b, g, f)\cdot a^*, 1, x).$$
We simplify the first element of the latter matrix.
\begin{eqnarray*}
((a,g,f) \cdot b, g, f)\cdot a^*&=&((a,g,f)\cdot b, g,f)\cdot \left(1,-\frac{1}{g(\bar{f})}, \bar{f}\right)\cdot a(x)\\
&=& \left(((a,g,f)\cdot b, g,f)\cdot 1, - g.\frac{1}{g(\bar{f}(f))}, \bar{f}(f)\right)\cdot a(x)\\
&=& ((a,g, f) \cdot b, -1, x)\cdot a(x)\\
&=& a_0 (a, g,f) \cdot b-x.1.\tilde{a}(x)\\
&=& (a,g,f)\cdot b -x \tilde{a}(x)\\
&=& b_0 a+xg \tilde{b}(f)-x \tilde{a}(x)\\
&=& a + x g \tilde{b}(f)-x \frac{a(x)-1}{x}\\
&=& a+xg \tilde{b}(f)-a(x)+1\\
&=& 1+xg \tilde{b}(f).\end{eqnarray*}
Finally, we have
$$(a,g,f) \cdot (b,1,x) \cdot (a, g, f)^{-1}=(1+xg \tilde{b}(f), 1, x).$$

We have the following canonical factorization.

\begin{equation}(a, g, f)=(a,1,x) \cdot (1, g, f).\end{equation}

This follows since
$$(a, 1, x)\cdot (1, g, f)=((a,1,x)\cdot 1, g, f)=(a, g, f).$$

Now let $(a, g, f) \in a\mathcal{R}$. We have
$$(a, g, f)\mathcal{N} = (1, g, f)\mathcal{N} \Leftrightarrow (1,g,f)^{-1}\cdot (a, g,f) \in \mathcal{N}.$$
Now
\begin{eqnarray*}
(1, g, f)^{-1} \cdot (a,g,f)&=& \left(1, \frac{1}{g \circ \bar{f}}, \bar{f}\right)\cdot (a, g, f)\\
&=& \left((a, g, f).1, \frac{1}{g \circ \bar{f}}.g(\bar{f}), f(\bar{f})\right)\\
&=& (a, 1, x) \in \mathcal{N}.\end{eqnarray*}
Hence modulo $\mathcal{N}$, we have
$$ (a, g,f) \backsim (1, g, f).$$
It is clear that we have a $1-1$ correspondence between almost-Riordan arrays of the form $(1, g, f)$ and Riordan arrays $(g,f)$.
Hence $$ a\mathcal{R}/\mathcal{N} = \mathcal{R}.$$

\begin{proposition} Let $a\mathcal{R}$ be the group of almost-Riordan arrays, $\mathcal{R}$ be the group of Riordan arrays, and $\mathcal{N}$ be the normal subgroup of $a\mathcal{R}$ consisting of arrays of the form $(a, 1, x)$ where $a_0=1$. Then
$$a\mathcal{R}/\mathcal{N} = \mathcal{R}.$$
\end{proposition}

\begin{proposition} $\mathcal{R}$ is not a normal subgroup of $a\mathcal{R}$.
\end{proposition}
\begin{proof}
We consider an element $\left(u, u\frac{v}{x}, v\right)$ of the subgroup $\mathcal{R}$. If $\mathcal{R}$ were a normal subgroup, then  for an arbitrary element $(a, g, f) \in a\mathcal{R}$, we would have
$$ (a, g, f)\cdot \left(u, u\frac{v}{x}, v\right) \cdot (a, g, f)^{-1}=\left(U, U \frac{V}{x}, V\right),$$ for appropriate power series $U(x)$ and $V(x)$.
Now we have
\begin{eqnarray*}(a, g, f)\cdot \left(u, u\frac{v}{x}, v\right)&=& \left((a,g,f)u, g u(f)\frac{v(f)}{f}, v(f)\right)\\
&=&\left(u_0 a+xg\tilde{u}(f),g u(f)\frac{v(f)}{f}, v(f)\right).\end{eqnarray*}
Hence
\begin{eqnarray*}
(a, g, f)\cdot \left(u, u\frac{v}{x}, v\right) \cdot (a, g, f)^{-1}&=&
\left(u_0 a+xg\tilde{u}(f),g u(f)\frac{v(f)}{f}, v(f)\right)\cdot (a, g, f)^{-1}\\
&=&\left(u_0 a+xg\tilde{u}(f),g u(f)\frac{v(f)}{f}, v(f)\right)\cdot \left(a^*,\frac{1}{g(\bar{f})},\bar{f}\right).
\end{eqnarray*}
This last expression is equal to
$$\left(\left(u_0 a+xg\tilde{u}(f),g u(f)\frac{v(f)}{f}, v(f)\right)\cdot a^*,gu(f)\frac{v(f)}{f} \frac{1}{g\circ \bar{f}(v(f))}, \bar{f}(v(f))\right).$$
We must therefore simplify $$\left(u_0 a+xg\tilde{u}(f),g u(f)\frac{v(f)}{f}, v(f)\right)\cdot a^*,$$ where
$$a^*=\left(1, -\frac{1}{g(\bar{f})}, \bar{f}\right)\cdot a.$$
Now
\begin{eqnarray*}
\left(u_0 a+xg\tilde{u}(f),g u(f)\frac{v(f)}{f}, v(f)\right)\cdot a^*&=&\left(u_0 a+xg\tilde{u}(f),g u(f)\frac{v(f)}{f}, v(f)\right) \cdot \left(1, -\frac{1}{g(\bar{f})}, \bar{f}\right)\cdot a\\
&=& \left(u_0 a+xg\tilde{u}(f), -gu(f)\frac{v(f)}{f} \frac{1}{g\circ \bar{f}(v(f))}, \bar{f}(v(f))\right)\cdot a\\
&=& a_0(u_0 a+xg\tilde{u}(f))-xgu(f)\frac{v(f)}{f} \frac{1}{g\circ \bar{f}(v(f))} \tilde{a}(\bar{f}(v(f)))\\
&=& a+xg\tilde{u}(f)-xgu(f)\frac{v(f)}{f} \frac{1}{g\circ \bar{f}(v(f))} \tilde{a}(\bar{f}(v(f))).
\end{eqnarray*}
We have thus arrived at
$$(a, g, f)\cdot \left(u, u\frac{v}{x}, v\right) \cdot (a, g, f)^{-1}=$$
$$\left(a+xg\tilde{u}(f)-xgu(f)\frac{v(f)}{f} \frac{1}{g\circ \bar{f}(v(f))} \tilde{a}(\bar{f}(v(f))),gu(f)\frac{v(f)}{f} \frac{1}{g\circ \bar{f}(v(f))}, \bar{f}(v(f))\right)=$$
$$\left(a+xg\tilde{u}(f)-xgu(f)\frac{v(f)}{f} \frac{1}{g\circ \bar{f}(v(f))} \frac{a(\bar{f}(v(f)))-1}{\bar{f}(v(f))},gu(f)\frac{v(f)}{f} \frac{1}{g\circ \bar{f}(v(f))}, \bar{f}(v(f))\right).$$
The first element expands to give
$$a+xg\tilde{u}(f)-xgu(f)\frac{v(f)}{f} \frac{1}{g\circ \bar{f}(v(f))} \frac{a(\bar{f}(v(f)))-1}{\bar{f}(v(f))}=$$
$$a+xg\tilde{u}(f)-xgu(f)\frac{v(f)}{f} \frac{1}{g\circ \bar{f}(v(f))}\frac{a(\bar{f}(v(f)))}{\bar{f}(v(f))}+xgu(f)\frac{v(f)}{f} \frac{1}{g\circ \bar{f}(v(f))}\frac{1}{\bar{f}(v(f))}.$$

Thus the obstruction to achieving normality is the expression
$$a+xg\tilde{u}(f)-xgu(f)\frac{v(f)}{f} \frac{1}{g\circ \bar{f}(v(f))}\frac{a(\bar{f}(v(f)))}{\bar{f}(v(f))}.$$
\end{proof}

The production matrix of $(a, g, f)$ is as follows.
\begin{proposition} The production matrix of the almost-Riordan array $(a, g, f)$ is given as follows. Its first column is $(a, g, f)^{-1}\cdot \tilde{a}(x)$. Its second column is given by
$(a, g, f)^{-1} \cdot g(x)$. Subsequent columns coincide with the $A$-sequence of the Riordan array $(g, f)$.
\end{proposition}
\begin{proof} This follows immediately from the definition of the production matrix
$$(a, g, f)^{-1} \cdot \overline{(a, g, f)}.$$
The matrix $\overline{(a, g, f)}$ has a first column generated by $\tilde{a}$, alongside the Riordan array
$(g, f)$. Thus the production matrix consists of the result of applying
$(a, g, f)^{-1}$ to $\tilde{a}$, alongside the result of applying the inverse $(g, f)^{-1}$ to $(g, f)$ without its first row (since in $(a, g, f)$, the component $(g, f)$ starts in the column that is one column in from the left).
\end{proof}

We call the first ($0$-th) column of the production matrix of $(a, g, f)$ the $\omega$ sequence, the second column the $Z$-sequence and the third column the $A$-sequence (with respective entries $\omega_n$, $Z_n$ and $A_n$).
Then we have the following.
$$T_{n,0}=\sum_{j=0}^n T_{n-1,j} \omega_k.$$
$$T_{n,1}=\sum_{j=0}^n T_{n-1,j} Z_j.$$
$$T_{n,k}=\sum_{j=k-1}^n T_{n-1, j} A_j, \quad k > 1.$$

We have
\begin{eqnarray*}
\omega(x)&=&(a, g, f)^{-1}\cdot \tilde{a}(x)\\
&=& \left(a^*, \frac{1}{g(\bar{f})}, \bar{f}\right)\cdot \tilde{a}(x) \\
&=& \tilde{a}_0 a^*+x\frac{1}{g(\bar{f})}\tilde{a}(\bar{f})\\
&=& a_1 a^*+x\frac{1}{g(\bar{f})}\tilde{a}(\bar{f})\\
&=& a_1 (a_0-x \frac{1}{g(\bar{f})} a(\bar{f}))+x\frac{1}{g(\bar{f})}\tilde{a}(\bar{f})\\
&=& a_1 + \frac{x}{g(\bar{f})}\left(\tilde{a}(\bar{f})-a_1 a(\bar{f})\right).
\end{eqnarray*}
\section{Iterating the process}
A natural question that arises is whether the process of adding a new column on the left can be iterated to assemble a hierarchy of higher groups? The following considerations indicate that this is indeed possible.

We define a set of matrices $\mathcal{R}^{(2)}$ as follows. Its elements are $4$-tuples of power series $(a, b, g, f)$ where $a$, $b$, $g$ and $f$ satisfy $a_0=1$, $b_0=1$, $g_0=1$ and $f_0=0, f_1=1$. We define a product of such elements by
\begin{equation}(a, b, g, f) \cdot (h, k, u, v)=((a, b, g, f)\cdot h, (b, g, f)\cdot k, g u(f), v(f)).\end{equation}
Here, the term $(b, g, f) \cdot k$ is to be taken in the sense of $a\mathcal{R}=\mathcal{R}^{(1)}$, while
the term $g u(f)=(g, u)\cdot f$ in the sense of $\mathcal{R}=\mathcal{R}^{(0)}$.
Thus
$$(a,b,g,f) \cdot (h, k,u, v)=((a, b,g,f,)\cdot h, (b,g,f)\cdot k, (g, u)\cdot f, v(f)).$$ It remains to say what is  $$(a, b, g, f) \cdot h$$. We define
\begin{equation}{\label{Eq}}(a, b, g, f) \cdot h = h_0 a + h_1 xb + x^2 g \tilde{\tilde{h}}(f),\end{equation} where
$$\tilde{\tilde{h}}=\frac{h(x)-h_0-h_1 x}{x^2}.$$
The $4$-tuple $(a, b, g, f)$ is identified with the following lower-triangular matrix: its first column (the $0$-th column) is given by the expansion of $a(x)$; the second column begins with a $0$, and from the $(1,1)$-position downwards coincides with the expansion of $b(x)$ (that is, the second column coincides with the expansion of $xb(x)$). Starting from the $(2,2)$ position, the matrix coincides with the Riordan array $(g,f)$. Other elements are zero. Matrix multiplication of a vector (when that vector's elements coincide with the expansion of a generating function $h(x)$) then corresponds to the rule given by Equation (\ref{Eq}). This can then be called the Fundamental Theorem for $\mathcal{R}^{(2)}$.

We can define the inverse of a $4$-tuple as follows.
\begin{equation}(a, b, g, f)^{-1}=\left(a^{**}, b^*, \frac{1}{g \circ \bar{f}}, \bar{f}\right),\end{equation}
where \begin{equation}b^*=(1, - g, f)^{-1}\cdot b = \left(1, -\frac{1}{g \circ \bar{f}}, \bar{f}\right)\cdot b,\end{equation}
and where
\begin{equation}a^{**} = (1, -b, -g, f)^{-1} \cdot a=\left(1, -b^*, -\frac{1}{g \circ \bar{f}}, \bar{f}\right) \cdot a.\end{equation}

\begin{example} We consider the element
$$(a,b, g, f)=\left(\frac{1-2}{1-3x}, \frac{1-x}{1-2}, \frac{1}{1-x}, \frac{x}{1-x}\right) \in \mathcal{R}^{(2)}.$$
We have
\begin{eqnarray*}b^*&=&\left(1, -\frac{1}{1-x}, \frac{x}{1-x}\right)^{-1}\cdot \frac{1-x}{1-2x}\\
&=& \left(1, -\frac{1}{1+x}, \frac{x}{1+x}\right)\cdot \frac{1-x}{1-2x}\\
&=& \frac{1-2x}{1-x}.\end{eqnarray*}
Then
\begin{eqnarray*}
a^{**}&=& (1, -b,-g, f)^{-1} \cdot a\\
&=& \left(1, -\frac{1-x}{1-2x}, -\frac{1}{1-x}, \frac{x}{1-x}\right)^{-1}\cdot a\\
&=& \left(1, -\frac{1-2x}{1-x}, -\frac{1}{1+x}, \frac{x}{1+x}\right) \cdot \frac{1-2x}{1-3x}\\
&=& 1 - x \frac{1-2x}{1-x}-x^2 \frac{1}{1+x} \widetilde{\widetilde{\left(\frac{1-2x}{1-3x}\right)}}\left(\frac{x}{1+x}\right)\\
&=& \frac{1-4x+3x^2-x^3}{(1-x)(1-2x)}.\end{eqnarray*}

Here, we have $$\widetilde{\widetilde{\left(\frac{1-2x}{1-3x}\right)}}=\frac{3}{1-3x}$$ and hence
$$\widetilde{\widetilde{\left(\frac{1-2x}{1-3x}\right)}}\left(\frac{x}{1+x}\right)=\frac{3}{1-3\frac{x}{1+x}}=\frac{3(1+x)}{1-2x}.$$
Thus we have
$$\left(\frac{1-2}{1-3x}, \frac{1-x}{1-2}, \frac{1}{1-x}, \frac{x}{1-x}\right)^{-1}=\left(\frac{1-4x+3x^2-x^3}{(1-x)(1-2x)}, \frac{1-2x}{1-x}, \frac{1}{1+x}, \frac{x}{1+x}\right).$$
\end{example}
\begin{example} There are many ways of constructing elements of $\mathcal{R}^{(2)}$. For instance, we can start with a Riordan array and pre-pend two columns appropriately. Alternatively, we could start with a Riordan array and multiply it by an element of the form $(a,b,1,x)$. The following shows another method of defining an element of $\mathcal{R}^{(2)}$, starting with an element of $\mathcal{R}^{(0)}=\mathcal{R}$.

We take the matrix $$\left(\frac{1}{1+x}, \frac{x}{(1+x)^2}\right)^{-1}=(c(x), c(x)-1), $$ which begins
$$\left(
\begin{array}{ccccccc}
 1 & 0 & 0 & 0 & 0 & 0 & 0 \\
 1 & 1 & 0 & 0 & 0 & 0 & 0 \\
 2 & 3 & 1 & 0 & 0 & 0 & 0 \\
 5 & 9 & 5 & 1 & 0 & 0 & 0 \\
 14 & 28 & 20 & 7 & 1 & 0 & 0 \\
 42 & 90 & 75 & 35 & 9 & 1 & 0 \\
 132 & 297 & 275 & 154 & 54 & 11 & 1 \\
\end{array}
\right).$$
We then form the product
$$\left(
\begin{array}{ccccccc}
 1 & 0 & 0 & 0 & 0 & 0 & 0 \\
 1 & 1 & 0 & 0 & 0 & 0 & 0 \\
 2 & 3 & 1 & 0 & 0 & 0 & 0 \\
 5 & 9 & 5 & 1 & 0 & 0 & 0 \\
 14 & 28 & 20 & 7 & 1 & 0 & 0 \\
 42 & 90 & 75 & 35 & 9 & 1 & 0 \\
 132 & 297 & 275 & 154 & 54 & 11 & 1 \\
\end{array}
\right)\cdot \left(
\begin{array}{ccccccccc}
 1 & 1 & 1 & 0 & 0 & 0 & 0 & 0 & 0 \\
 0 & 1 & 1 & 1 & 0 & 0 & 0 & 0 & 0 \\
 0 & 0 & 1 & 1 & 1 & 0 & 0 & 0 & 0 \\
 0 & 0 & 0 & 1 & 1 & 1 & 0 & 0 & 0 \\
 0 & 0 & 0 & 0 & 1 & 1 & 1 & 0 & 0 \\
 0 & 0 & 0 & 0 & 0 & 1 & 1 & 1 & 0 \\
 0 & 0 & 0 & 0 & 0 & 0 & 1 & 1 & 1 \\
\end{array}
\right),$$ to obtain the matrix that begins
$$\left(
\begin{array}{ccccccccc}
 1 & 1 & 1 & 0 & 0 & 0 & 0 & 0 & 0 \\
 1 & 2 & 2 & 1 & 0 & 0 & 0 & 0 & 0 \\
 2 & 5 & 6 & 4 & 1 & 0 & 0 & 0 & 0 \\
 5 & 14 & 19 & 15 & 6 & 1 & 0 & 0 & 0 \\
 14 & 42 & 62 & 55 & 28 & 8 & 1 & 0 & 0 \\
 42 & 132 & 207 & 200 & 119 & 45 & 10 & 1 & 0
   \\
 132 & 429 & 704 & 726 & 483 & 219 & 66 & 12 &
   1 \\
\end{array}
\right).$$
We now complete this matrix to be lower-triangular as follows
$$M=\left(
\begin{array}{ccccccccc}
 1 & 0 & 0 & 0 & 0 & 0 & 0 & 0 & 0 \\
 1 & 1 & 0 & 0 & 0 & 0 & 0 & 0 & 0 \\
 1 & 1 & 1 & 0 & 0 & 0 & 0 & 0 & 0 \\
 1 & 2 & 2 & 1 & 0 & 0 & 0 & 0 & 0 \\
 2 & 5 & 6 & 4 & 1 & 0 & 0 & 0 & 0 \\
 5 & 14 & 19 & 15 & 6 & 1 & 0 & 0 & 0 \\
 14 & 42 & 62 & 55 & 28 & 8 & 1 & 0 & 0 \\
 42 & 132 & 207 & 200 & 119 & 45 & 10 & 1 & 0
   \\
 132 & 429 & 704 & 726 & 483 & 219 & 66 & 12 &
   1 \\
\end{array}
\right).$$
The production matrix of this array then begins
$$\left(
\begin{array}{cccccccc}
 1 & 1 & 0 & 0 & 0 & 0 & 0 & 0 \\
 0 & 0 & 1 & 0 & 0 & 0 & 0 & 0 \\
 0 & 1 & 1 & 1 & 0 & 0 & 0 & 0 \\
 1 & 2 & 2 & 2 & 1 & 0 & 0 & 0 \\
 -1 & -2 & 0 & 1 & 2 & 1 & 0 & 0 \\
 0 & 0 & -1 & 0 & 1 & 2 & 1 & 0 \\
 1 & 2 & 1 & 0 & 0 & 1 & 2 & 1 \\
 -1 & -2 & 0 & 0 & 0 & 0 & 1 & 2 \\
\end{array}
\right),$$ indicating that the matrix $M$ is an element of $\mathcal{R}^{(2)}$.

We note for instance that the transform $b_n$ of the Fibonacci numbers $F_n$ by this matrix has a Hankel transform
with generating function $$\frac{-x(1-10x+24x^2+64x^3)}{(1-4x)^4},$$ while the Hankel transform of $b_{n+1}$ has generating function $$\frac{1-6x}{(1-4x)^2}.$$
\end{example}

We  have the following proposition.
\begin{proposition} The set of $4$-tuples $\mathcal{R}^{(2)}$ defined above is a group, with identity
$I=(1,1,1,x)$. The subset $\mathcal{N}^{(2)}$  of $4$-tuples of the form $(a, b, 1, x)$ is a normal subgroup of $\mathcal{R}^{(2)}$ and we have
$$  \mathcal{R}^{(2)}/\mathcal{N}^{(2)} = \mathcal{R}^{(0)}.$$
\end{proposition}

In similar fashion, we may define a hierarchy of sets of $n$-tuples of power series $\mathcal{R}^{(n-2)}$, where $\mathcal{R}^{(0)}=\mathcal{R},$ the Riordan group.

\section{Example 1: Almost-Riordan arrays and a special transformation}
In \cite{Basor}, the authors consider a transformation on sequences $a_n$ with the property $a_{-n}=a_n$, defined by
$$b_n=\sum_{k=0}^{n-1} \binom{n-1}{k}(a_{1-n+2k}+a_{2-n+k}).$$
For the special sequence $a_n=x^n$ for $n \ge 0$, $a_n=x^{-n}$ for $n<0$ (i.e. $a_n=x^{|n|}$), we obtain that the images
$$b_0, b_1, b_2, b_3, b_4,b_5, \ldots$$ are given by
$$0, x + 1, x^2 + 2x + 1, x^3 + 2x^2 + 3x + 2, x^4 + 2x^3 + 4x^2 + 6x + 3, x^5 + 2x^4 + 5x^3 + 8x^2 + 10x + 6, \ldots,$$
with a coefficient array which begins
$$\left(
\begin{array}{ccccccc}
 0 & 0 & 0 & 0 & 0 & 0 & 0 \\
 1 & 1 & 0 & 0 & 0 & 0 & 0 \\
 1 & 2 & 1 & 0 & 0 & 0 & 0 \\
 2 & 3 & 2 & 1 & 0 & 0 & 0 \\
 3 & 6 & 4 & 2 & 1 & 0 & 0 \\
 6 & 10 & 8 & 5 & 2 & 1 & 0 \\
 10 & 20 & 15 & 10 & 6 & 2 & 1 \\
\end{array}
\right).$$

Since the sequences $a_n$ that are of interest in this case all have $a_0=0$, we can equivalently use the array that begins

$$M=\left(
\begin{array}{ccccccc}
 1 & 0 & 0 & 0 & 0 & 0 & 0 \\
 1 & 1 & 0 & 0 & 0 & 0 & 0 \\
 1 & 2 & 1 & 0 & 0 & 0 & 0 \\
 2 & 3 & 2 & 1 & 0 & 0 & 0 \\
 3 & 6 & 4 & 2 & 1 & 0 & 0 \\
 6 & 10 & 8 & 5 & 2 & 1 & 0 \\
 10 & 20 & 15 & 10 & 6 & 2 & 1 \\
\end{array}
\right).$$

This is an almost-Riordan array, defined by
$$\left(\frac{1+2x+\sqrt{1-4x^2}}{2 \sqrt{1-4x^2}}, \frac{(1+2x)c(x^2)}{\sqrt{1-4x^2}}, xc(x^2)\right)=\left(\frac{1}{1+x}, \frac{1-x}{1+x+x^2+x^3}, \frac{x}{1+x^2}\right)^{-1},$$  where
$c(x)=\frac{1-\sqrt{1-4x^2}}{2x}$ is the generating function of the Catalan numbers.
Its first column is given by $\binom{n-1}{\lfloor \frac{n}{2} \rfloor}$. Other than for the first column, this coincides with the Riordan array
$$R=\left(\frac{1+2x}{\sqrt{1-4x^2}}, xc(x^2)\right)=\left(\frac{1-x}{1+x}, \frac{x}{1+x^2}\right)^{-1},$$ which has first column
$$1,2,2,4,6,10,20,\ldots.$$
In fact, we have
$$M = R \cdot \left(
\begin{array}{ccccccc}
 1 & 0 & 0 & 0 & 0 & 0 & 0 \\
 -1 & 1 & 0 & 0 & 0 & 0 & 0 \\
 1 & 0 & 1 & 0 & 0 & 0 & 0 \\
 -1 & 0 & 0 & 1 & 0 & 0 & 0 \\
 1 & 0 & 0 & 0 & 1 & 0 & 0 \\
 -1 & 0 & 0 & 0 & 0 & 1 & 0 \\
1 & 0 & 0 & 0 & 0 & 0 & 1 \\
\end{array}
\right)=R \cdot \left(\frac{1}{1+x}, 1, x\right).$$

The production array of the  almost-Riordan array $M$ begins
$$\left(
\begin{array}{ccccccc}
 1 & 1 & 0 & 0 & 0 & 0 & 0 \\
 0 & 1 & 1 & 0 & 0 & 0 & 0 \\
 1 & 0 & 0 & 1 & 0 & 0 & 0 \\
 -1 & 1 & 1 & 0 & 1 & 0 & 0 \\
1 & -1 & 0 & 1 & 0 & 1 & 0 \\
 -1 & 1 & 0 & 0 & 0 & 1 & 0 \\
 1 & -1 & 0 & 0 & 0 & 0 & 1 \\
\end{array}
\right),$$ where we can see the $\omega-$, $Z-$ and $A-$sequences.

We note that the matrix
$$\left(
\begin{array}{ccccccc}
 1 & 1 & 0 & 0 & 0 & 0 & 0 \\
 1 & 2 & 1 & 0 & 0 & 0 & 0 \\
 2 & 3 & 2 & 1 & 0 & 0 & 0 \\
 3 & 6 & 4 & 2 & 1 & 0 & 0 \\
 6 & 10 & 8 & 5 & 2 & 1 & 0 \\
 10 & 20 & 15 & 10 & 6 & 2 & 1 \\
 20 & 35 & 30 & 21 & 12 & 7 & 2 \\
\end{array}
\right)$$

is equal to
$$\left(
\begin{array}{ccccccc}
 1 & 0 & 0 & 0 & 0 & 0 & 0 \\
 1 & 1 & 0 & 0 & 0 & 0 & 0 \\
 2 & 1 & 1 & 0 & 0 & 0 & 0 \\
 3 & 3 & 1 & 1 & 0 & 0 & 0 \\
 6 & 4 & 4 & 1 & 1 & 0 & 0 \\
 10 & 10 & 5 & 5 & 1 & 1 & 0 \\
 20 & 15 & 15 & 6 & 6 & 1 & 1 \\
\end{array}
\right)\cdot \left(
\begin{array}{ccccccc}
 1 & 1 & 0 & 0 & 0 & 0 & 0 \\
 0 & 1 & 1 & 0 & 0 & 0 & 0 \\
 0 & 0 & 1 & 1 & 0 & 0 & 0 \\
 0 & 0 & 0 & 1 & 1 & 0 & 0 \\
 0 & 0 & 0 & 0 & 1 & 1 & 0 \\
 0 & 0 & 0 & 0 & 0 & 1 & 1 \\
 0 & 0 & 0 & 0 & 0 & 0 & 1 \\
\end{array}
\right),$$ where the first matrix in this product is the Riordan array
$$\left(\frac{1-x}{1+x^2}, \frac{x}{1+x^2}\right)^{-1}=\left(\frac{1+xc(x^2)}{\sqrt{1-4x^2}}, xc(x^2)\right).$$
Thus the product is given by
$$\left(\frac{1+xc(x^2)}{\sqrt{1-4x^2}}, xc(x^2)\right) \cdot (1+x,x)^t.$$

We next multiply the coefficient array by the binomial matrix $B=\left(\binom{n}{k}\right)$ to obtain the almost-Riordan array that begins
$$B\cdot M=\left(
\begin{array}{ccccccc}
 1 & 0 & 0 & 0 & 0 & 0 & 0 \\
 2 & 1 & 0 & 0 & 0 & 0 & 0 \\
 4 & 4 & 1 & 0 & 0 & 0 & 0 \\
 9 & 12 & 5 & 1 & 0 & 0 & 0 \\
 22 & 34 & 18 & 6 & 1 & 0 & 0 \\
 57 & 95 & 58 & 25 & 7 & 1 & 0 \\
 153 & 266 & 178 & 90 & 33 & 8 & 1 \\
\end{array}
\right).$$

This is the almost-Riordan array
$$\left(\frac{1+x+\sqrt{1-2x-3x^2}}{2(1-x)\sqrt{1-2x-3x^2}}, \frac{(1-x)\sqrt{1-2x-3x^2}-(1-2x-3x^2)}{2(1-4x+3x^2)}, \frac{1-x-\sqrt{1-2x-3x^2}}{2x^2}\right),$$ where the last entry is the g.f. of the Motzkin numbers. This almost-Riordan array has an inverse given by
$$\left(\frac{1+2x^2}{1+2x+2x^2+x^3}, \frac{1-x+x^2-x^3}{(1+x)(1+x+x^2)^2}, \frac{x}{1+x+x^2}\right).$$

The production array of the above almost-Riordan array begins
$$\left(
\begin{array}{ccccccc}
 2 & 1 & 0 & 0 & 0 & 0 & 0 \\
 0 & 2 & 1 & 0 & 0 & 0 & 0 \\
 1 & 0 & 1 & 1 & 0 & 0 & 0 \\
 -1 & 1 & 1 & 1 & 1 & 0 & 0 \\
1 & -1 & 0 & 1 & 1 & 1 & 0 \\
 -1 & 1 & 0 & 0 & 0 & 1 & 1 \\
 1 & -1 & 0 & 0 & 0 & 0 & 1 \\
\end{array}
\right),$$ where we can see the usual effect of the binomial transform on the diagonal elements (namely, we increment each diagonal element by $1$).

We finally multiply the almost-Riordan array $M$ by
$$\left(
\begin{array}{ccccccc}
 1 & 0 & 0 & 0 & 0 & 0 & 0 \\
 -1 & 1 & 0 & 0 & 0 & 0 & 0 \\
 -1 & 0 & 1 & 0 & 0 & 0 & 0 \\
 -2 & 0 & 0 & 1 & 0 & 0 & 0 \\
 -3 & 0 & 0 & 0 & 1 & 0 & 0 \\
 -6 & 0 & 0 & 0 & 0 & 1 & 0 \\
 -10 & 0 & 0 & 0 & 0 & 0 & 1 \\
\end{array}
\right)=\left(
\begin{array}{ccccccc}
 1 & 0 & 0 & 0 & 0 & 0 & 0 \\
 1 & 1 & 0 & 0 & 0 & 0 & 0 \\
 1 & 0 & 1 & 0 & 0 & 0 & 0 \\
 2 & 0 & 0 & 1 & 0 & 0 & 0 \\
 3 & 0 & 0 & 0 & 1 & 0 & 0 \\
 6 & 0 & 0 & 0 & 0 & 1 & 0 \\
 10 & 0 & 0 & 0 & 0 & 0 & 1 \\
\end{array}
\right)^{-1}$$ to obtain
$$\left(
\begin{array}{ccccccc}
 1 & 0 & 0 & 0 & 0 & 0 & 0 \\
 -1 & 1 & 0 & 0 & 0 & 0 & 0 \\
 -1 & 0 & 1 & 0 & 0 & 0 & 0 \\
 -2 & 0 & 0 & 1 & 0 & 0 & 0 \\
 -3 & 0 & 0 & 0 & 1 & 0 & 0 \\
 -6 & 0 & 0 & 0 & 0 & 1 & 0 \\
 -10 & 0 & 0 & 0 & 0 & 0 & 1 \\
\end{array}
\right)\cdot \left(
\begin{array}{ccccccc}
 1 & 0 & 0 & 0 & 0 & 0 & 0 \\
 1 & 1 & 0 & 0 & 0 & 0 & 0 \\
 1 & 2 & 1 & 0 & 0 & 0 & 0 \\
 2 & 3 & 2 & 1 & 0 & 0 & 0 \\
 3 & 6 & 4 & 2 & 1 & 0 & 0 \\
 6 & 10 & 8 & 5 & 2 & 1 & 0 \\
 10 & 20 & 15 & 10 & 6 & 2 & 1 \\
\end{array}
\right)$$
$$=\left(
\begin{array}{ccccccc}
 1 & 0 & 0 & 0 & 0 & 0 & 0 \\
 0 & 1 & 0 & 0 & 0 & 0 & 0 \\
 0 & 2 & 1 & 0 & 0 & 0 & 0 \\
 0 & 3 & 2 & 1 & 0 & 0 & 0 \\
 0 & 6 & 4 & 2 & 1 & 0 & 0 \\
 0 & 10 & 8 & 5 & 2 & 1 & 0 \\
0 & 20 & 15 & 10 & 6 & 2 & 1 \\
\end{array}
\right).$$

\section{Example 2: Some Catalan related almost-Riordan arrays and Hankel transforms}

We base this section on the almost-Riordan array given by
$$T = \left(
\begin{array}{ccccccc}
 1 & 0 & 0 & 0 & 0 & 0 & 0 \\
 -1 & 1 & 0 & 0 & 0 & 0 & 0 \\
 1 & 0 & 1 & 0 & 0 & 0 & 0 \\
 -1 & 0 & 0 & 1 & 0 & 0 & 0 \\
 1 & 0 & 0 & 0 & 1 & 0 & 0 \\
 -1 & 0 & 0 & 0 & 0 & 1 & 0 \\
1 & 0 & 0 & 0 & 0 & 0 & 1 \\
\end{array}
\right) \cdot R,$$ where $R$ is the Riordan array
$$R=\left(\frac{1+2x}{\sqrt{1-4x^2}}, xc(x^2)\right)=\left(\frac{1-x}{1+x}, \frac{x}{1+x^2}\right)^{-1}$$ seen in the previous section.
We obtain that $T$ is the almost-Riordan array that begins
$$T=\left(
\begin{array}{ccccccc}
 1 & 0 & 0 & 0 & 0 & 0 & 0 \\
 1 & 1 & 0 & 0 & 0 & 0 & 0 \\
 3 & 2 & 1 & 0 & 0 & 0 & 0 \\
 3 & 3 & 2 & 1 & 0 & 0 & 0 \\
 7 & 6 & 4 & 2 & 1 & 0 & 0 \\
 11 & 10 & 8 & 5 & 2 & 1 & 0 \\
 21 & 20 & 15 & 10 & 6 & 2 & 1 \\
\end{array}
\right),$$
where the first column
$$1,1,3,3,7,11,21,\ldots $$ is given by
$$a_n=(-1)^n+2 \binom{n-1}{\lfloor \frac{n-1}{2} \rfloor},$$ with generating function
$$\frac{1}{1+x}+\frac{1+2x-\sqrt{1-4x^2}}{\sqrt{1-4x^2}}.$$ We note that this sequence has a Hankel transform that begins
$$1, 2, -4, -24, 64, 352, -64, -1664, 256, 7680, -4096,\ldots$$ with a (conjectured) generating function

$$\frac{1+2x-16x^3+48x^4+256x^5+256x^6-128x^7}{(1+4x^2)^2(1-4x^2+16x^4)}.$$
In fact, $T$ is the almost-Riordan array
$$\left(\frac{1}{1+x}+\frac{1+2x-\sqrt{1-4x^2}}{\sqrt{1-4x^2}}, \frac{(1+2x)c(x^2)}{ \sqrt{1-4x^2}},xc(x^2)\right),$$ where the Riordan array
$$\left(\frac{(1+2x)c(x^2)}{ \sqrt{1-4x^2}},xc(x^2)\right)$$ has the inverse
$$\left(\frac{1-x}{(1+x)(1+x^2)}, \frac{x}{1+x^2}\right).$$
We are interested in the effect of the matrix $T$ on the Fibonacci polynomials
$$F_n(y)=\sum_{k=0}^{n-1} \binom{n-k-1}{k}y^k,$$ which have generating function
$$\frac{x}{1-x-yx^2}.$$
We find that the Hankel transform of the image $b_n(y)$ of $F_n(y)$ by $T$ has generating function
$$-\frac{x(1-2(y-2)x+(y-1)^2 x^2)}{1-2(y^2-2y-1)x^2+(y-1)^4x^4}.$$
Regarded as a bi-variate generating function, this generates the array that begins
$$G=\left(
\begin{array}{ccccccccccc}
 0 & 0 & 0 & 0 & 0 & 0 & 0 & 0 & 0 & 0 & 0 \\
 1 & 0 & 0 & 0 & 0 & 0 & 0 & 0 & 0 & 0 & 0 \\
 4 & -2 & 0 & 0 & 0 & 0 & 0 & 0 & 0 & 0 & 0 \\
 -1 & -6 & 3 & 0 & 0 & 0 & 0 & 0 & 0 & 0 & 0 \\
 -8 & -12 & 16 & -4 & 0 & 0 & 0 & 0 & 0 & 0 & 0 \\
 1 & 20 & 10 & -20 & 5 & 0 & 0 & 0 & 0 & 0 & 0 \\
 12 & 74 & -32 & -52 & 36 & -6 & 0 & 0 & 0 & 0 & 0
   \\
 -1 & -42 & -119 & 84 & 49 & -42 & 7 & 0 & 0 & 0 & 0
   \\
 -16 & -216 & -224 & 488 & -80 & -136 & 64 & -8 & 0
   & 0 & 0 \\
 1 & 72 & 468 & 168 & -738 & 216 & 132 & -72 & 9 & 0
   & 0 \\
 20 & 470 & 1536 & -984 & -2008 & 1828 & -160 & -280
   & 100 & -10 & 0 \\
\end{array}
\right).$$
Thus for instance $G$ applied to $2^n$ (i.e. to the vector $<1,2,4,8,\ldots>$), returns the Hankel transform of the image by $T$ of the Jacobsthal numbers, namely the sequence
$$0,-1,0,1,0,-1,0,1,0,-1,0,\ldots.$$
The Hankel transform of the once-shifted sequence $b_{n+1}(y)$ is also of interest. The sequence $b_{n+1}(y)$ is the image of the shifted Fibonacci polynomial $\sum_{k=0}^n \binom{n-k}{k}y^k$ by the Riordan array
$$\left(\frac{(1+2x)c(x^2)}{ \sqrt{1-4x^2}},xc(x^2)\right).$$

We find that the Hankel transform of $b_{n+1}(y)$ is generated by
$$\frac{1+(y-1)x}{1+2x+(y-1)x^2}$$ with coefficient array
$$\left(
\begin{array}{ccccccccccc}
 1 & 0 & 0 & 0 & 0 & 0 & 0 & 0 & 0 & 0 & 0 \\
 -3 & 1 & 0 & 0 & 0 & 0 & 0 & 0 & 0 & 0 & 0 \\
 5 & 0 & -1 & 0 & 0 & 0 & 0 & 0 & 0 & 0 & 0 \\
 -7 & -7 & 7 & -1 & 0 & 0 & 0 & 0 & 0 & 0 & 0 \\
 9 & 24 & -18 & 0 & 1 & 0 & 0 & 0 & 0 & 0 & 0 \\
 -11 & -55 & 22 & 22 & -11 & 1 & 0 & 0 & 0 & 0 & 0
   \\
 13 & 104 & 13 & -104 & 39 & 0 & -1 & 0 & 0 & 0 & 0
   \\
 -15 & -175 & -147 & 285 & -45 & -45 & 15 & -1 & 0 &
   0 & 0 \\
 17 & 272 & 476 & -544 & -170 & 272 & -68 & 0 & 1 &
   0 & 0 \\
 -19 & -399 & -1140 & 684 & 1102 & -874 & 76 & 76 &
   -19 & 1 & 0 \\
 21 & 560 & 2331 & -144 & -3598 & 1680 & 630 & -560
   & 105 & 0 & -1 \\
\end{array}
\right).$$
The Hankel transform of the shifted image of the Fibonacci numbers (case of $y=1$) is then given by $(-2)^n$ while that of the Jacobsthal numbers is $(-1)^n$.

We note an interesting property of the previous coefficient array. If we multiply it (on the left) by the binomial matrix $B$, we obtain the array that begins
$$\left(
\begin{array}{ccccccccccc}
 1 & 0 & 0 & 0 & 0 & 0 & 0 & 0 & 0 & 0 & 0 \\
 -2 & 1 & 0 & 0 & 0 & 0 & 0 & 0 & 0 & 0 & 0 \\
 0 & 2 & -1 & 0 & 0 & 0 & 0 & 0 & 0 & 0 & 0 \\
 0 & -4 & 4 & -1 & 0 & 0 & 0 & 0 & 0 & 0 & 0 \\
 0 & 0 & 4 & -4 & 1 & 0 & 0 & 0 & 0 & 0 & 0 \\
 0 & 0 & -8 & 12 & -6 & 1 & 0 & 0 & 0 & 0 & 0 \\
 0 & 0 & 0 & 8 & -12 & 6 & -1 & 0 & 0 & 0 & 0 \\
 0 & 0 & 0 & -16 & 32 & -24 & 8 & -1 & 0 & 0 & 0 \\
 0 & 0 & 0 & 0 & 16 & -32 & 24 & -8 & 1 & 0 & 0 \\
 0 & 0 & 0 & 0 & -32 & 80 & -80 & 40 & -10 & 1 & 0
   \\
 0 & 0 & 0 & 0 & 0 & 32 & -80 & 80 & -40 & 10 & -1
   \\
\end{array}
\right).$$

We see embedded in this the square $B^2$ of the binomial matrix, which begins
$$\left(
\begin{array}{ccccccc}
 1 & 0 & 0 & 0 & 0 & 0 & 0 \\
 2 & 1 & 0 & 0 & 0 & 0 & 0 \\
 4 & 4 & 1 & 0 & 0 & 0 & 0 \\
 8 & 12 & 6 & 1 & 0 & 0 & 0 \\
 16 & 32 & 24 & 8 & 1 & 0 & 0 \\
 32 & 80 & 80 & 40 & 10 & 1 & 0 \\
 64 & 192 & 240 & 160 & 60 & 12 & 1 \\
\end{array}
\right).$$

Reading the columns of the transformed coefficient array in reverse order (from the bottom up, left to right), we obtain an array associated to the Chebyshev polynomials of the fourth kind (see \seqnum{A228565} and \seqnum{A180870}).

We finish by looking at the almost-Riordan array
$$(2-c(x), c(x), xc(x)) = \left(\frac{1}{1-x}, (1-x), x(1-x)\right)^{-1}, $$ which begins
$$\left(
\begin{array}{ccccccc}
 1 & 0 & 0 & 0 & 0 & 0 & 0 \\
 -1 & 1 & 0 & 0 & 0 & 0 & 0 \\
 -2 & 1 & 1 & 0 & 0 & 0 & 0 \\
 -5 & 2 & 2 & 1 & 0 & 0 & 0 \\
 -14 & 5 & 5 & 3 & 1 & 0 & 0 \\
 -42 & 14 & 14 & 9 & 4 & 1 & 0 \\
 -132 & 42 & 42 & 28 & 14 & 5 & 1 \\
\end{array}
\right).$$
Applying this to the shifted Fibonacci polynomials
$$F_{n+1}(y)=\sum_{k=0}^{\lfloor \frac{n}{2} \rfloor} \binom{n-k}{k}y^k,$$ we get a sequence that begins
$$1, 0, y, 4y, y^2 + 14y, 7y^2 + 48y, y^3 + 35y^2 + 165y, 10y^3 + 154y^2 + 572y, \ldots.$$ Taking the Hankel transform of this sequence, we arrive at a Hankel transform $H_n(y)$ that has the property that
$\frac{H_n(y)}{y^n}$ begins
$$1, 1, -2, 2(y - 1) - y^2, 2y^3 - 7y^2 + y + 3, - 3y^4 + 16y^3 - 4y^2 - 5y + 3, 4y^5 - 29y^4 + 25y^3 + 34y^2 - 5y - 4,\ldots.$$ This sequence has coefficient array that begins
$$\left(
\begin{array}{ccccccc}
 1 & 0 & 0 & 0 & 0 & 0 & 0 \\
 1 & 0 & 0 & 0 & 0 & 0 & 0 \\
 -2 & 0 & 0 & 0 & 0 & 0 & 0 \\
 -2 & 2 & -1 & 0 & 0 & 0 & 0 \\
 3 & 1 & -7 & 2 & 0 & 0 & 0 \\
 3 & -5 & -4 & 16 & -3 & 0 & 0 \\
 -4 & -5 & 34 & 25 & -29 & 4 & 0 \\
\end{array}
\right),$$ and has a (conjectured) generating function of
\begin{scriptsize}
$$\frac{1+(2y+1)x+(y^2+6y+2)x^2+2(2y^2+4y+1)x^3+(5y^2+3y+1)x^4+(y+1)(2y+1)x^5-yx^6-yx^7}{(1+(y+1)x+(2y+2)x^2+(y+1)x^3+x^4)^2}.$$ \end{scriptsize}
\section{Conclusion}
We have shown that the Riordan group of invertible lower-triangular matrices is isomorphic to a subgroup of another group of lower-triangular invertible matrices. These matrics, called ``almost-Riordan'' arrays in this note, appear to be worthy of study in their own right. We have identified one normal subgroup of this new group. We have shown instances where elements of the new group can produce interesting transformations on sequences, and in particular the images of such transformations, for suitable starting sequences, may have significant Hankel determinants.

\bigskip
\hrule
\bigskip
\noindent 2010 {\it Mathematics Subject Classification}: Primary
15B36; Secondary 11B83, 11C20.
\noindent \emph{Keywords:} Riordan group, Riordan array, almost-Riordan group, almost-Riordan array.

\end{document}